\documentclass[12pt,amsfonts]{article}
\usepackage[margin=1in]{geometry}  
\usepackage{graphicx}              
\usepackage{amsmath}               
\usepackage{amsfonts}              
\usepackage{amsthm}                
\usepackage{graphicx,color} 
\def\nd{\noindent}
\def\thend{\rule{3mm}{3mm}}
\def\R{\mathbb{R}}
\def\D{\mathcal{D}}

\def\N{\mathcal{N}}

\def\L{\mathcal{L}}
\newtheorem{claim}{Claim}[section]
\newtheorem{thm}{Theorem}[section]
\newtheorem{prop}{Proposition}[section]
\newtheorem{lem}{Lemma}[section]

\numberwithin{equation}{section}

\begin{document}

\title{Existence of a heteroclinic solution for a double well potential equation in an infinite cylinder of $\mathbb{R}^N$  }
\author{\sf Claudianor O. Alves\thanks{Research of C. O. Alves partially supported by  CNPq 304036/2013-7  and INCTMAT/CNPq/Brazil} 
\\
\small{Universidade Federal de Campina Grande, }\\
\small{Unidade Acad\^emica de Matem\'atica } \\
\small{CEP: 58429-900 - Campina Grande-PB, Brazil}\\
\small{ e-mail: coalves@mat.ufcg.edu}\\
\vspace{1mm}\\
}
\date{}
\maketitle
\begin{abstract}
This paper concerns with the existence of a heteroclinic solution for the following class of elliptic equations 
$$
-\Delta{u}+A(\epsilon x, y)V'(u)=0, \quad \mbox{in} \quad \Omega,
$$  
where $\epsilon >0$, $\Omega=\R \times \D$ is an infinite cylinder of $\mathbb{R}^N$ with $N \geq 2$.  Here, we have considered a large class of potential $V$ that includes the Ginzburg-Landau potential $V(t)=(t^{2}-1)^{2}$ and two geometric conditions on the  function $A$. In the first condition we assume that $A$  is asymptotic at infinity to a periodic function, while in the second one $A$ satisfies 
$$
0<A_0=A(0,y)=\inf_{(x,y) \in \Omega}A(x,y) <  \liminf_{|(x,y)| \to +\infty}A(x,y)=A_\infty<\infty, \quad \forall y \in \D. 
$$ 
\end{abstract}

\vspace{0.5 cm}


\section{Introduction}

This paper concerns with the existence of a heteroclinic solution for the following class of elliptic equations
\begin{equation} \label{E01}
-\Delta u + A(\epsilon x, y)V'(u)=0,   \quad \mbox{in} \quad \Omega, \tag{PDE}
\end{equation}
together with the Neumann boundary condition
\begin{equation} \label{E02}
\frac{\partial u}{\partial \nu}(x,y)=0, \ x \in \R, \ y \in \partial \D, \tag{NC}
\end{equation}
where $N \geq 2$, $\epsilon>0$, $\Omega$ is an infinite cylinder of the type $\Omega=\R \times \D$  with $\D \subset \R^{N-1}$ being a smooth bounded domain and $\nu=\nu(y)$ is the normal vector outward pointing to $\partial \D.$ Related to the functions $A:\overline{\Omega} \to \mathbb{R}$ and $V:\mathbb{R} \to \mathbb{R}$, we are assuming the following conditions: \\

\noindent {\bf Conditions on $V$:} \\

\noindent $ (V_1) \,\, V \in C^{1}(\mathbb{R}, \mathbb{R})$. \\

\noindent $ (V_2)$ \,\, $V(-1)=V(1)=0$ and $V(t) \geq 0 \quad \mbox{for all} \quad  t \in \mathbb{R}$,  \\

\noindent and \\

\noindent $ (V_3) \,\,\,   V(t)>0$ \,\, for all $ t \not= -1,1.$ \\

An example of  $ V$ satisfying $ (V_1)-(V_3)$ is the Ginzburg-Landau potential $V(t)=(t^{2}-1)^{2}.$

\vspace{0.5cm}

\noindent {\bf Conditions on $A$:} \\

In whole this paper $A$ is a $C^{1}$-function that belongs to one  of the following classes:  \\

\noindent {\bf Class 1:} \, $A$ is  asymptotic at infinity to a periodic function. \\

In this class, we assume that there exists a $C^1$-function $A_p: \overline{\Omega} \to \mathbb{R}$, which is $1-$ periodic in $x$, such that   
$$
|A(x,y)-A_p(x,y)| \to 0 \,\,\, \mbox{as} \,\,\,\ |(x,y)| \to +\infty \leqno{(A_1)}
$$ 
and
$$
0<A_0=\inf_{(x,y) \in \Omega}A(x,y)\leq A(x,y)<A_p(x,y), \,\,\, \forall (x,y) \in \Omega.  \leqno{(A_2)}
$$

This type of condition is well known when we are working with periodic asymptotically problem of the type  
$$
-\Delta u +A(x)u=f(u), \,\,\,\, \mbox{in} \quad \mathbb{R}^{N},
$$
see for example Alves, Carri\~ao and Miyagaki \cite{ACM}, Jianfu and Xiping \cite{JX} and their references.

\vspace{0.5 cm}

\noindent {\bf Class 2:} \, $A$ satisfies the Rabinowitz's condition. \\

In this class of functions, we suppose that 
$$
0<A_0=A(0,y)=\inf_{(x,y) \in \Omega}A(x,y) <  \liminf_{|(x,y)| \to +\infty}A(x,y)=A_\infty<\infty, \quad \forall y \in \D. \leqno{(A_3)}
$$ 

A condition like above has been introduced by Rabinowitz \cite[Theorem 4.33]{R12} to study the existence of solution for a P.D.E. of the type 
$$
-\epsilon^{2} \Delta u +A(x)u=f(u) \,\,\,\, \mbox{in} \quad \mathbb{R}^{N},
$$
where $\epsilon>0$, $f:\mathbb{R} \to \mathbb{R}$ is a continuous function with subcritical growth and $A:\mathbb{R}^N \to \mathbb{R}$ is a continuous function satisfying 
$$
0<\inf_{x \in \mathbb{R}^N}A(x) < \liminf_{|x| \to \infty}A(x).
$$
By using variational methods, more precisely the mountain pass theorem, Rabinowitz has established the existence of solution for $\epsilon$ small enough. For this reason, throughout this article,  we will call $(A_3)$ of Rabinowitz's condition.

By $(V_1)-(V_3),$ $V$ is a double well potential and we are interested in the existence of solutions for (\ref{E01}) and (\ref{E02}) that are heteroclinic in $x$ from 1 to -1. 
A heteroclinic solution from 1 to -1 is a function $u \in C^{2}(\overline{\Omega}, \mathbb{R})$ verifying (\ref{E01})-(\ref{E02})
with
$$
u(x,y) \to 1 \quad \mbox{as} \quad x \to -\infty \quad \mbox{and} \quad  u(x,y) \to -1 \quad \mbox{as} \quad x \to +\infty, \quad \mbox{uniformly in} \quad y \in \D.
$$

In \cite{RabinowitzTokyo}, Rabinowitz  has proved the existence of a heteroclinic solution for elliptic equations of the type 
$$
-\Delta u=g(x,y,u), \quad \mbox{in} \quad \Omega,
$$
together with the boundary condition (NC) and also with the Dirichlet boundary condition, that is, 
\begin{equation}
u(x,y)=0, \quad x \in \mathbb{R}, \,  y \in \partial \D.  \tag{DC}
\end{equation}
In order to prove the existence of heteroclinic solution, in Section 2,  Rabinowitz has used variational methods by supposing on  $g$ the conditions below: \\
\noindent $(g_1) \quad g \in C^{1}(\overline{\Omega} \times \mathbb{R}, \mathbb{R})$. \\
\noindent $(g_2) \quad g(x,y,t)$ is even and 1-periodic in $x$. 

\noindent In Section 3, Rabinowitz has considered some conditions on $g$  that permit to study other classes of nonlinearity. From these comments, we see that if 
$$
g(x,y,t)=A(x,y)V'(t), \eqno{(g)}
$$
Rabinowitz has studied the case when $A(x,y)$ is 1-periodic in $x$, see Section 2 of the paper above mentioned. Here, we continue this study, because we will work with two new classes of function $A$ that were not considered in that paper, more precisely the Classes $1$ and $2$.

After Byeon, Montecchiari and Rabinowitz \cite{Byeon} have established the existence of heteroclinic solution $u:\Omega \to \mathbb{R}^m$ for a large class of elliptic system like 
$$
-\Delta u+V_u(x,u)=0, \quad \mbox{in} \quad \Omega,
$$
together with the boundary condition (NC) by supposing the following conditions on potential $V$: \\
$(H_1)$ \,\, $V \in C^{1}(\overline{\Omega} \times \mathbb{R}^{m}, \mathbb{R})$ and $V(x_1+1,x_2....,x_N,y)=V(x,y),$ i.e., $V$ is 1-periodic in $x_1$. \\
$(H_2)$ \,\,There are points $a^- \not= a^+$ such that $V(x,a^{\pm})=0$  for all $x \in \Omega$ and $V(x,y)>0$ otherwise. \\
$(H_3)$ \,\, There is a constant $\underline{V}>0$ such that $\displaystyle \liminf_{|t| \to \infty}V(x,t) \geq \underline{V}$ uniformly in $x \in \Omega$. \\
$(H_4)$ \,\, For $N \geq 2$, there exist constants $c_1;C_1 > 0$ such that
$$
|V_u(x,t)|\leq c_1+C_1|t|^p, 
$$ 
where $1<p<\frac{N+2}{N-2}$ for $N \geq 3$ and there is no upper growth restriction on $p$ if $N = 2$.  In the present paper, we are working with the potential $V(x,y,u)=A(x,y)V(u)$, with $A$ belonging to Classes $1$ or $2$ and $V$ satisfying $(H_1)-(H_4)$. Our paper also continues the study made in \cite{Byeon} for $m=1$, because we are working with other classes  of function $A$. Here, it is very important to mention that the study of elliptic system as above is very subtle because some arguments used for the scalar case $m=1$ cannot be used for general case $m>1$ as for example maximum principle.

In the literature we also find interesting papers that study the existence of heteroclinic solution for elliptic equations in whole $\mathbb{R}^N$ like
$$
-\Delta u(x,y)+ A(x,y)V'(u(x,y))=0, \quad (x,y) \in \mathbb{R}^N,
$$
by supposing different conditions on $A$ and $V$, see for example,  Alessio and Montecchiari \cite{AM1}, Alessio, Jeanjean and Montecchiari \cite{AlessioJM2}, Alessio, Gui and Montecchiari \cite{AlessioGP},  Rabinowitz \cite{R1}, Rabinowitz and Stredulinsky
\cite{RabStr-0,RabStr-1,RabStr-2} and their references. The reader can find versions for elliptic systems of the above equation in Alama, Bronsard and Gui \cite{Alama}, Alessio, Jeanjean and Montecchiari \cite{AlessioJM2}, Montecchiari and Rabinowitz \cite{MonteRab2016} and references therein.

Motived by papers \cite{Byeon} and \cite{RabinowitzTokyo}, we intend to establish the existence of a heteroclinic solution for the equation  (\ref{E01}) under the Neumann boundary conditions by working with the Classes 1 and 2. As in the above papers, we have used variational method, more precisely minimization technical on a special set, however new ideas have been introduced in the study of the problem, see for example, Proposition \ref{crucial*} in Section 2. The regularity and behavior of the heteroclinic are obtained by using the same arguments found in \cite{Byeon}.  

\vspace{0.5 cm}

Our main results are the following 

\begin{thm} \label{T1} Assume $(V_1)-(V_3)$, $\epsilon =1$ and that $A$ belongs to Class 1. Then problem (\ref{E01})-(\ref{E02}) has a heteroclinic solution from $1$ to $-1$.  
\end{thm}

\begin{thm} \label{T22} Assume $(V_1)-(V_3)$ and that $A$ belongs to Class 2.  Then, there is $\epsilon_0 >0$ such that  problem (\ref{E01})-(\ref{E02}) possesses a heteroclinic solution from $1$ to $-1$ for all $\epsilon \in (0, \epsilon_0).$  
\end{thm}

The plan of the paper is as follows: In Section 2, we prove some technical results, which will be useful to prove the above theorems. In Section 3 we prove the Theorem \ref{T1}, while in Section 4 we prove the Theorem \ref{T22}.

\vspace{0.5 cm}

\section{Preliminary Results}
Consider the problem (\ref{E01})-(\ref{E02}) with $\epsilon =1$, more precisely,

$$
\left\{  \begin{array}{r} -\Delta u +  {A}(x,y)V'(u)=0,\quad  \  \forall \ (x,y)\in \Omega=\R \times \D,\\
\mbox{}\\
\displaystyle \frac{\partial u}{\partial \nu}(x,y)=0, \quad \forall x \in \R,\ y \in \partial \D. \\
\end{array}\right.
$$
	
In the sequel, we define the set   
\begin{eqnarray}\label{Gamma}\Gamma&=& \{U \in W_{loc}^{1,2}(\Omega): |\nabla U| \in L^{2}(\Omega), \,\, \|P_kU -1\|_{L^{2}(\Omega_1)}\to 0 \quad \mbox{as} \\ && k \to -\infty \quad  \mbox{and} \quad \|P_kU +1\|_{L^{2}(\Omega_1)} \to 0 \quad \mbox{as} \quad \quad k \to +\infty \}, \nonumber
\end{eqnarray}
where $\Omega_1=(0,1) \times \D$ and
$$
P_k U (x,y)=U(x+k,y), \quad \mbox{for} \quad (x,y) \in \Omega \quad \mbox{and} \quad  k \in \mathbb{Z}.
$$

It is very important to observe that $\Gamma \not= \emptyset $, because the function $\Phi$ given by 
\begin{equation} \label{phi}
\Phi(x,y)=
\left\{
\begin{array}{lcl}
1, & \mbox{if}&\quad x \leq j, y \in \D, \\\
2j+1-2x, & \mbox{if}&\quad j < x \leq j+1, y \in \D, \\
-1, & \mbox{if}& \quad j+1 < x, y \in \D,
\end{array}
\right.
\end{equation}
belongs to $\Gamma$. Furthermore, we also fix 
$$
\L(u)=\frac{1}{2}|\nabla u|^2 + {A}(x,y)V(u),
$$
and the functionals $J:\Gamma \rightarrow \R \cup \{+\infty\} $ given by
\begin{equation} \label{J}
J(U)=\sum_{k \in \mathbb{Z}} I_k (U) 
\end{equation}
and $I_k: W^{1,2}((k,k+1) \times \D)\rightarrow \R$ defined by
$$
I_k(U)=\int_{k}^{k+1} \int_{\D} \L(U) dx dy.
$$
Associated with functional $J$ we have the number 
\begin{equation}
\label{nivel1}
\Theta^*=\inf\{J(U)\,:\,U \in \Gamma\}.
\end{equation}
By (\ref{phi}), $\Phi \in \Gamma$,  then $\Theta^* <+\infty.$ From definition of $\Theta^*$, there exists a minimizing sequence $(U_n) \subset \Gamma$ for $J$, that is, 
\begin{equation} \label{J}
J(U_n) \to \Theta^* \ \mbox{as}\ n \rightarrow  \infty.
\end{equation}
Without loss of generality, we can assume that $(U_n)$ verifies
\begin{equation} \label{limitacao}
-1 \leq U_n(x,y) \leq 1, \quad \forall (x,y) \in \Omega.
\end{equation}
Indeed, for each $n \in \mathbb{N}$ let us consider 
$$
\tilde{U}_n(x,y)=
\left\{
\begin{array}{rcl}
-1, & \mbox{if} & U_{n}(x,y)\leq -1, \\
U_n(x,y), & \mbox{if} & -1 \leq U_n(x,y) \leq 1, \\
1, & \mbox{if} & U_n(x,y) \geq 1.
\end{array}
\right. 
$$ 
It is easy to check that $\tilde{U}_n \in W^{1,2}_{loc}(\Omega)$ with 
$$
|\tilde{U}_n(x,y)-1|\leq |U_n(x,y)-1|, \quad \forall (x,y) \in \Omega
$$ 
and
$$
|\tilde{U}_n(x,y)+1|\leq |U_n(x,y)+1|, \quad \forall (x,y) \in \Omega.
$$ 
Hence $(\tilde{U}_n) \subset \Gamma$, and so, 
$$
\Theta^* \leq J(\tilde{U}_n), \quad \forall n \in \mathbb{N}. 
$$
Since 
$$
J(\tilde{U}_n) \leq J(U_n), \quad \forall n \in \mathbb{N}, 
$$
it follows that 
$$
\Theta^* \leq J(\tilde{U}_n) \leq J(U_n)=\Theta^*+o_n(1),
$$
thereby showing that $(\tilde{U}_n)$ is also a minimizing sequence for $J$ on $\Gamma$ with
$$
-1 \leq \tilde{U}_n(x,y) \leq 1, \quad \forall (x,y) \in \Omega.
$$ 
From (\ref{J})-(\ref{limitacao}), there is $M>0$ independent of $k$ and $m$ such that
$$
\||\nabla U_m|\|_{L^{2}((k,k+1)\times \D)}+\|U_m\|_{L^{2}((k,k+1)\times \D)}\leq M, \quad \forall m \in \mathbb{N} \quad \mbox{and} \quad k \in \mathbb{Z}.  
$$
Consequently, $(U_n)$ is bounded in $E_k=W^{1,2}((k,k+1)\times \D)$, endowed with the usual norm, for all $k  \in \mathbb{Z}$. Then for some subsequence, there is $U \in W_{loc}^{1,2}(\Omega)$ such that
\begin{equation} \label{EQT1}
U_n \rightharpoonup U \quad \mbox{in} \quad E_k, \quad \forall k \in \mathbb{Z},
\end{equation}
\begin{equation} \label{EQT2*}
	U_n \to U \quad \mbox{in} \quad L^{2}((k,k+1) \times \D), \quad \forall k \in \mathbb{Z},
\end{equation}
and
\begin{equation} \label{EQT2}
U_n(x,y) \to U(x,y), \quad \mbox{a.e. in} \quad \Omega.
\end{equation}
Therefore, from (\ref{J})-(\ref{EQT2}),
\begin{equation} \label{E3.1}
J(U) \leq \Theta^{*} \quad \mbox{and} \quad -1 \leq U(x,y) \leq 1, \quad \mbox{a.e. in } \quad \Omega.
\end{equation}

\vspace{0.5 cm}

In the next section, our main goal is proving that $U$ is the desired heteroclinic solution, and in this point, the conditions on function $A$ play their role. However, before doing that we need to say that if  $A$ is $1$- periodic in $x$, the same arguments explored in \cite{Byeon} guarantee the existence of a heteroclinic solution  $W^*$ from $1$ to $-1$.

\section{Proof of Theorem \ref{T1}: $A$ is  asymptotic at infinity to a periodic function}

By hypothesis,   
	$$
	{A}(x,y) < {A}_p(x,y), \quad \forall (x,y) \in \Omega.
	$$
Then, if $W^* \in \Gamma$ is a heteroclinic solution for the periodic case, we must have   
$$
\Theta^* \leq J(W^*) < J_p(W^*)=\Theta_p^*,
$$
that is,
\begin{equation} \label{niveis}
	\Theta^* < \Theta_p^*.
\end{equation}
The last inequality will be a key point in our approach. In what follows, $(U_m) \subset \Gamma$ is a minimizing sequence for $J$ with 
$$
-1\leq U_m(x,y) \leq 1, \quad \forall (x,y) \in \Omega.
$$

By using the fact that $(U_n) \subset \Gamma$, given $\tau \in (0, \sqrt{|\Omega_1|})$, for each  $m \in \mathbb{N}$, there are $k_1(m), k_2(m) \in \mathbb{Z}$ such that 
\begin{equation} \label{E0}
\|P_{-j}Q_m - 1\|_{L^{2}(\Omega_1)} < \tau, \quad \|Q_m - 1\|_{L^{2}(\Omega_1)} \geq \tau
\end{equation}
and
\begin{equation} \label{E00}
\|P_{j}R_m + 1\|_{L^{2}(\Omega_1)} < \tau, \quad \|R_m + 1\|_{L^{2}(\Omega_1)} \geq \tau
\end{equation}
for all $j \in \mathbb{N}$, where
$$
Q_m(x,y)=U_{m}(x-k_1(m),y) \quad \mbox{and} \quad R_m(x,y)=U_{m}(x+k_2(m),y), \quad \forall (x,y) \in \Omega.
$$

The reader is invited to observe that only one of the following conditions holds:\\

\noindent $(I)$ \,\,\,\,\,\,\,\,\,  $k_1(m), k_2(m)<0$ \quad for some subsequence. \\
\noindent $(II)$ \,\,\,\,\,\, $k_1(m) >0$ for all $m \in \mathbb{N}$, and $k_2(m)<0$ for some subsequence. \\
\noindent $(III)$ \,\,\, $k_2(m) >0$ for all $m \in \mathbb{N}$, and $k_1(m)<0$ for some subsequence. \\
\noindent $(IV)$ \,\,\, $k_1(m),k_2(m) >0$ for all $m \in \mathbb{N}$.  \\

The above conditions are crucial to prove the following proposition

\begin{prop} (Main proposition ) \label{crucial*} There is $j_0 \in \mathbb{N}$ such that 
\begin{equation} \label{NEWLIMIT0}
\|U-1\|_{L^{2}((-j,-j+1) \times \D)}\leq \tau \quad \mbox{and} \quad  \|U+1\|_{L^{2}((j,j+1) \times \D)}\leq \tau, \quad \forall j \geq j_0.
\end{equation}   
\end{prop}

We will assume for a moment that Proposition \ref{crucial*} is proved and show Theorem \ref{T1}.

\vspace{0.5 cm}

\noindent {\bf Proof of Theorem \ref{T1}} 
From the limit $J(U_n) \to \Theta^*$, we get
\begin{equation} \label{theta2}
\sum_{j \in \mathbb{Z}} I_j (U)=J(U) \leq \Theta^*, 
\end{equation}
from where it follows that
$$
I_j (U) \to 0 \quad \mbox{as} \quad j \to -\infty,
$$
or equivalently
$$
\int_{0}^{1}\int_{\D}|\nabla P_{-j}U|^{2}\,dxdy+\int_{0}^{1}\int_{\D}A(x,y)V(P_{-j}U)\,dxdy \to 0 \quad \mbox{as} \quad j \to -\infty.
$$
As $U \in L^{\infty}(\Omega)$, we have that $(P_{-j}U)$ is a bounded sequence in $W^{1,2}(\Omega_1)$. Thus, there is a subsequence $(P_{-j_k}U)$ of $(P_{-j}U)$ and $\hat{U} \in W^{1,2}(\Omega_1)$ such that  
$$
P_{-j_k}U \rightharpoonup \hat{U} \quad \mbox{in} \quad W^{1,2}(\Omega_1) \quad \mbox{as} \quad j_k \to +\infty
$$
$$
P_{-j_k}U \to  \hat{U} \quad \mbox{in} \quad L^{2}(\Omega_1) \quad \mbox{as} \quad j_k \to +\infty
$$
and
$$
P_{-j_k}U(x,y) \to \hat{U}(x,y) \quad \mbox{a.e. in} \quad \Omega_1, \quad \mbox{as} \quad j_k \to +\infty.
$$
From this,  
$$
\int_{0}^{1}\int_{\D}V(\hat{U})\,dxdy=0,
$$
then 
$$
\hat{U}=1 \quad \mbox{or} \quad \hat{U}=-1, 
$$
and so,
$$
P_{-j_k}U \to 1 \quad \mbox{or} \quad P_{-j_k}U \to -1 \quad \mbox{in} \quad L^{2}(\Omega_1) \quad \mbox{as} \quad j_k \to +\infty.
$$
Since $\tau \in (0, \sqrt{|\Omega_1|})$, these limits combine with (\ref{NEWLIMIT0}) to give   
$$
P_{-j_k}U \to 1 \quad \mbox{in} \quad L^{2}(\Omega_1) \quad \mbox{as} \quad j_k \to +\infty.
$$
The above argument also yields
$$
P_{-j}U \to 1 \quad \mbox{in} \quad L^{2}(\Omega_1) \quad \mbox{as} \quad j \to +\infty.
$$
Similar reasoning proves
$$
P_{j}U \to -1 \quad \mbox{in} \quad L^{2}(\Omega_1) \quad \mbox{as} \quad j \to +\infty.
$$
Consequently, $U \in \Gamma$ and $-1 \leq U(x,y) \leq 1$ for all $(x,y) \in \overline{\Omega}$. Moreover,  by (\ref{theta2}), 
$$
J(U)=\Theta^*.
$$
Now, we claim that for each $\phi \in C_{0}^{\infty}(\overline{\Omega})$, we have $\frac{\partial J}{\partial \phi}(U)=0$, where $\frac{\partial J}{\partial \phi}(U)$  denotes the  directional derivative of $J$ at $U$ in the direction of $\phi \in C_{0}^{\infty}(\overline{\Omega})$, where  
$$
C_{0}^{\infty}(\overline{\Omega})=\{\phi:\overline{\Omega}\to \mathbb{R}\,:\; \exists \psi \in C_{0}^{\infty}(\mathbb{R}^{N},\mathbb{R}) \,\, \mbox{such that} \,\, \psi(x)=\phi(x), \,\,\, \forall x \in \overline{\Omega} \}.
$$

Indeed, taking $w=U + t\phi$ with $\phi \in C_{0}^{\infty}(\overline{\Omega})$ and $t \in \mathbb{R}$, we derive that for $k$ large  enough, let's say, $|k| > \ell_0,$ we have
$$
I_k (U+ t\phi)=I_k (U),\ \forall |k| > \ell_0.
$$
Thereby
$$
\frac{J(U + t\phi)-J(U)}{t}= \frac{1}{t}(\sum_{k \in \mathbb{Z}} (I_k (U+ t\phi)-I_k (U))=\sum_{k= -\ell_0}^{\ell_0} \left(\frac{I_k (U+ t\phi)-I_k (U)}{t}\right),
$$
and so,  
$$
\frac{\partial J}{\partial \phi}(U)=\lim_{t \to 0}\frac{J(U + t\phi)-J(U)}{t}=\sum_{k= -\ell_0}^{\ell_0} I'_k(U)\phi.
$$
As $w \in \Gamma$ and $J(U) \leq J(w)$, a standard argument ensures that $\frac{\partial J}{\partial \phi}(U)=0,$ for all $\phi \in C_{0}^{\infty}(\overline{\Omega})$. Therefore, 
$$
\int_{\Omega}\nabla U \nabla \phi\, dx + \int_{\Omega}A(x,y)V(U)\phi \,dx=0, \quad \forall \phi \in C_{0}^{\infty}(\overline{\Omega}).
$$
From this, $U$ is a weak solution of (PDE). A regularity argument from \cite[Section 6]{Byeon} implies that $U \in C^{2}(\overline{\Omega}, \mathbb{R})$, and that $U$ is a classical solution of  
$$
-\Delta U + A(x,y)V'(U)=0, \quad \mbox{in} \quad \Omega \quad  \mbox{and} \quad \frac{\partial U}{\partial \nu}=0, \ x \in \R, \ y \in \partial \D,
$$
with
$$
U(x,y) \to 1 \quad \mbox{as} \quad x \to -\infty \quad \mbox{and} \quad  U(x,y) \to -1 \quad \mbox{as} \quad x \to +\infty, \quad \mbox{uniformly in} \quad y \in \D.
$$
From this, $U$ is a heteroclinic solution from 1 to -1, which finishes the proof of Theorem \ref{T1}. $\Box$

\vspace{0.5 cm}

\noindent {\bf Proof of Proposition \ref{crucial*} } Note that if $(I)$ holds, then 
$$
\|U_m-1\|_{L^{2}((-j-1,-j) \times \D)} \leq \tau \quad \mbox{and} \quad  \|U_m+1\|_{L^{2}((j,j+1) \times \D)} \leq \tau, \quad \forall j,m \in \mathbb{N}. 
$$
This together with (\ref{EQT2*}) yield
$$
\|U-1\|_{L^{2}((-j-1,-j) \times \D)} \leq \tau \quad \mbox{and} \quad  \|U+1\|_{L^{2}((j,j+1) \times \D)} \leq \tau, \quad \forall j \in \mathbb{N},
$$
and the proposition follows with $j_0=0$.

Now, we will prove the proposition by supposing that $(II)$ holds. To begin with, we make the following claim  
\begin{claim} \label{NOVACLAIM}
 $(k_1(m))$ is bounded.
\end{claim}
In what follows, let us denote $(k_1(m))$ by $(k(m))$. Assume by contradiction that there is a subsequence of $(k(m))$, still denoted by itself, with $k(m) \to +\infty$. The boundedness of $(U_m)$ in $E_k$, implies that $(Q_m)$ is also bounded in 	$E_k$. Thus, for some subsequence, there is $W \in W_{loc}^{1,2}(\Omega)$ such that
\begin{equation} \label{E1}
Q_m \rightharpoonup W \quad \mbox{in} \quad E_k, \quad \forall k \in \mathbb{Z},
\end{equation}
\begin{equation} \label{E1'}
Q_m \to W  \mbox{in} \quad L^{2}((k,k+1) \times \D), \quad \forall k \in \mathbb{Z},
\end{equation}

\begin{equation} \label{E2}
Q_m(x,y) \to W(x,y), \quad \mbox{a.e. in} \quad \Omega,
\end{equation}
and
\begin{equation} \label{E2'}
-1 \leq  W(x,y) \leq 1, \quad \mbox{a.e. in} \quad \Omega.
\end{equation}

A simples change of variables gives us
\begin{equation} \label{M1}
\sum_{k \in \mathbb{Z}} \tilde{I}_k (Q_m)= J(U_m)=\Theta^{*}+o_m(1)\leq  \Theta^{*}+1
\end{equation}
where 
$$
\tilde{I}_k(U)=\int_{k}^{k+1} \int_{\D} \tilde{\L}_m(U) dx dy
$$
and
$$
\tilde{\L}_m(u)=\frac{1}{2}|\nabla u|^2 + A(x-k(m),y)V(u).
$$
Consequently, the Fatou's Lemma together with $(A_1)$ and (\ref{E1})-(\ref{M1}) provides 
\begin{equation} \label{E3}
J_{p}(W) \leq \Theta^{*},
\end{equation}	
which gives
\begin{equation} \label{EZZZ3}
I_{p,-j}(W) \to 0 \quad \mbox{as} \quad j \to +\infty.
\end{equation}
Setting for each $j \in \mathbb{N}$ the function $\widetilde{W}_j=P_{-j}W$, the fact that  $W \in L^{\infty}(\Omega)$ together with the Sobolev embeddings guarantee the existence of $W_0 \in L^{2}(\Omega_1)$,  and a subsequence of $(\widetilde{W}_j)$, still denoted by itself,  such that 
$$
\widetilde{W}_j \to  W_0 \quad \mbox{in} \quad L^{2}(\Omega_1),
$$
that is,
\begin{equation} \label{E5}
\|\widetilde{W}_j - W_0\|_{L^{2}(\Omega_1)} \to 0.
\end{equation}
This limit and (\ref{E0}) lead to
$$
\|W_0 - 1\|_{L^{2}(\Omega_1)} \leq \tau.
$$
On the other hand, by (\ref{EZZZ3}),
$$
I_{p,0}(W_0)=0, 
$$
from where it follows that $W_0=1$ or $W_0=-1$. 
As $ \tau \in (0,\sqrt{|\Omega_1|})$, we must have $W_0=1$. Thereby, 
\begin{equation} \label{E50}
\|\widetilde{W}_j - 1\|_{L^{2}(\Omega_1)} \to 0 \quad \mbox{as} \quad j \to +\infty.
\end{equation}

Now, fixing $W_j=P_jW$ for $j \in \mathbb{N}$, the same reasoning works to show that there exists $\widehat{W}_0 \in L^{2}(\Omega_1)$  and a subsequence of $(W_j)$, still denoted by itself,  such that 
\begin{equation} \label{W1}
W_j \to \widehat{W}_0 \,\, \mbox{in} \,\,  L^{2}(\Omega_1) \quad \mbox{as} \quad j \to \infty,
\end{equation}
or equivalently,
\begin{equation} \label{E5}
\|W_j - \widehat{W}_0\|_{L^{2}(\Omega_1)} \to 0.
\end{equation}
This information gathering with the limit 
$$
I_{p,j}(W) \to 0 \quad \mbox{as} \quad j \to +\infty
$$
leads to $\widehat{W}_0=1$ or $\widehat{W}_0=-1$. Next we are going to show that $\widehat{W}_0=-1$. To see why, assume by contradiction that $\widehat{W}_0=1$. From  (\ref{E0}) and (\ref{W1}), there is $j_1 \in \mathbb{N}$ such that
$$
\|W-1\|_{L^{2}((j_1-1,j_1) \times \D)}  \geq \tau \quad \mbox{and} \quad \|W-1\|_{L^{2}((j_1,j_1+1) \times \D)}  \leq \tau.
$$
As $Q_m \to W$ in $L^{2}((j_1-1,j_1+1) \times \D)$, there is $m_0 \in \mathbb{N}$ satisfying  
$$
\|Q_m-1\|_{L^{2}((j_1-1,j_1) \times \D)}  \geq \tau/2 \quad \mbox{and} \quad \|Q_m-1\|_{L^{2}((j_1,j_1+1) \times \D)}  \leq 2\tau, \quad \forall m \geq m_0.
$$
In what follows, we denoted by $\beta=\beta(\tau)$ the real number given by 
$$
\frac{\beta}{\tilde{A}_0}=\inf_{u \in \N_{\tau}}I_{*,\tau}(u), 
$$
where $\tilde{A}_0=\min\{1,A_0\}$,  
$$
\N_{\tau}=\{u \in W^{1,2}((-1,1) \times \D)\,:\, \|u-1\|_{L^{2}((-1,0) \times \D)}  \geq \tau/2 \quad \mbox{and} \quad \|u-1\|_{L^{2}((0,1) \times \D)}  \leq 2\tau \}
$$
and $I_{*,\tau}: W^{1,2}((-1,1) \times \D)\rightarrow \R$ is defined by
$$
I_{*,\tau}(u)=\int_{-1}^{1} \int_{\D}(|\nabla u|^{2}+V(u))\,\,  dx dy.
$$
Hence, by a simple change of variable
\begin{equation} \label{BETA}
\int_{j_1-1}^{j_1+1} \int_{\D} (|\nabla Q_m|^{2}+V(Q_m))\,\,  dxdy \geq \frac{\beta}{\tilde{A}_0}, \quad \forall m \geq m_0.
\end{equation}
Here we would like point out that the same arguments found in \cite[Proposition 2.14]{Byeon} work to show that $\beta>0$.  Having this in mind, we can assume without loss of generality that
\begin{equation} \label{Beta0}
J(U_m) \leq \Theta^* + \beta/4, \quad \forall m \geq m_0.
\end{equation}
In the sequel, for each $j \geq j_1+2$ and $m \geq m_0$, let us consider the function
$$
Z_{j,m}(x,y)=
\left\{
\begin{array}{lcl}
1, & \mbox{if}&\quad x \leq j, y \in \D, \\\
((j+1)-x)+(x-j)Q_m(x,y), & \mbox{if}&\quad j < x \leq j+1, y \in \D, \\
Q_m(x,y), & \mbox{if}& \quad j+1 < x, y \in \D.
\end{array}
\right.
$$
By a direct computation, we see that $Z_{j,m} \in \Gamma$ and 
$$
J_p(Z_{j,m})=I_{p,j}(Z_{j,m})+\sum_{k=j+1}^{\infty}I_{p,k}(Q_m)=I_{p,j}(Z_{j,m})+\sum_{k=j+1+k(m)}^{\infty}I_{p,k}(U_m),
$$
and so,
$$
\Theta_{p}^{*} \leq J_p(Z_{j,m})= I_{p,j}(Z_{j,m})+\sum_{k=j+1+k(m)}^{\infty}I_{p,k}(U_m).
$$
As $A$ verifies $(A_1)-(A_2)$ and $(J(U_n))$ is bounded, increasing  $m_0$ if necessary, we have 
\begin{equation} \label{E6B}
\sum_{k=j+1+k(m)}^{\infty}I_{p,k}(U_m) \leq \sum_{k=j+1+k(m)}^{\infty}I_{k}(U_m) + \beta/4, \quad \forall m \geq m_0.
\end{equation}
Now, as $j \geq j_1+2$, (\ref{E6B})  implies in the inequality  
$$
\Theta_{p}^{*} \leq I_{p,j}(Z_{j,m})+ J(U_m) - \tilde{A}_0 \int_{j_1-1}^{j_1+1} \int_{\D} (|\nabla Q_m|^{2}+V(Q_m))\,\,dxdy + \beta/4,  
$$
which combine with (\ref{BETA})-(\ref{Beta0}) to give  
\begin{equation} \label{E6}
\Theta_{p}^{*} \leq I_{p,j}(Z_{j,m}) + \Theta^* +\frac{\beta}{4} -\tilde{A}_0 \frac{\beta}{\tilde{A}_0}+\frac{\beta}{4}=I_{p,j}(Z_{j,m}) + \Theta^*-  \frac{\beta}{2}.
\end{equation}
Since
$$
-1 \leq W_j(x,y) \leq 1 \,\,\, \mbox{and} \,\,\,\, W_j \to 1 \,\, \mbox{in} \,\ W^{1,2}({\Omega_1}) \quad \mbox{as} \quad j \to +\infty, 
$$
it is easy to check that
$$
\lim_{j \to +\infty}\int_{0}^{1}\int_{\D}A(x+j,y)V((-x+1+xW_j)\,dxdy=0 
$$
and
$$
\lim_{j \to +\infty}\int_{0}^{1}\int_{\D}|1-W_j|^{2}\,dxdy=0.
$$
Thus, given $\delta >0$, there is  $j_0=j_0(\delta)>j_1+2$, which is independent of $m$, such that
\begin{equation} \label{NOVAESTIMATIVA}
\int_{0}^{1}\int_{\D}A(x+j,y)V(-x+1+xW_j)\,dxdy < \delta, \,\,\,\,\,\, \forall j \geq j_0
\end{equation}
and
\begin{equation} \label{NOVAESTIMATIVA2}
\int_{0}^{1}\int_{\D}|1-W_j|^{2}\,dxdy< \delta, \,\,\,\,\,\, \forall j \geq j_0.
\end{equation}

To continue, we further claim there is $j=j(m) \geq j_0$ and $m \geq m_0$ such that  
\begin{equation} \label{NOVAE}
	I_{p,j}(Z_{j,m})=\int_{j}^{j+1} \int_{\D} \L_p(Z_{j,m}) dx dy < \beta/2.
\end{equation}
If the claim does not hold, for each $j \geq j_0$, there exists $m_1=m_1(j) \geq m_0$ verifying 
$$
\int_{j}^{j+1} \int_{\D} \L_p(Z_{j,m}) dx dy  \geq \beta/2, \quad \forall m \geq m_1.
$$
From definition of $Z_{j,m}$ and $(A_2)$,
$$
\int_{j}^{j+1}|\nabla Z_{j,m}|^{2}\,dxdx \geq \beta/2 - \int_{j}^{j+1} \int_{\D}A(x,y)V((j+1)-x +(x-j)Q_m)\,dxdy.
$$
Recalling that
$$
\lim_{m \to +\infty}\int_{j}^{j+1} \int_{\D}A(x,y)V((j+1)-x +(x-j)Q_m)\,dxdy =\int_{0}^{1}\int_{\D}A(x+j,y)V((-x+1+xW_j)\,dxdy < \delta,
$$
for $j \geq j_0$ and $\delta < \beta/ 4$, there exists $m_2=m_2(j) \geq m_1(j)$ such that 
$$
\int_{j}^{j+1}\int_{\D}|\nabla Z_{j,m}|^{2}\,dxdx \geq \beta/4, \quad \forall m \geq m_2.
$$
Using again the definition of $Z_{j,m}$, there is a constant $C>0$ such that 
$$
\int_{j}^{j+1}\int_{\D}|\nabla Z_{j,m}|^{2}\,dxdy  \leq C\left(\int_{0}^{1}\int_{\D}|1-P_j(Q_m)|^{2}\,dxdy + \int_{j}^{j+1}\int_{\D}|\nabla Q_m|^{2}\,dxdy \right).
$$ 
Now, fixing  $\delta < \frac{\beta}{8C}$ in  (\ref{NOVAESTIMATIVA2}), we obtain
$$
\int_{j}^{j+1}\int_{\D}|\nabla Q_m|^{2}\,dxdy \geq \beta/8, \quad \forall m \geq m_2(j).
$$
Let $l \in \mathbb{N}$ such that
$$
(l+1)\beta/8 > \Theta^{*}+1
$$
and fix $m>\max\{m_2(j): j_0 \leq j \leq j_0+l\}$. Then, 
$$
\sum_{k \in \mathbb{Z}} \tilde{I}_k (Q_m) \geq \Theta^{*}+1,
$$
which contradicts (\ref{M1}), thereby showing (\ref{NOVAE}). Thus, by (\ref{E6}) and (\ref{NOVAE}),
$$
\Theta_{p}^{*} < \Theta^{*},
$$
contrary to (\ref{niveis}), and this ensures that $\widehat{W}_0=-1$. From the above study, we deduce that $W \in \Gamma$, then by  (\ref{E3}), 
$$
\Theta_{p}^{*} \leq J(W) \leq \Theta^{*},
$$
which is absurd. This proves the Claim \ref{NOVACLAIM}. Hence, there is a subsequence $(k_1(m))$, still denoted by itself, and $k_* \in \mathbb{N}$ such that
$k_1(m)=k_*$ for all $m$ in $\mathbb{N}$. Hence, the inequality below
$$
\|P_{-j}Q_m - 1\|_{L^{2}(\Omega_1)} < \tau, \quad \forall m \in \mathbb{N} \quad \mbox{and} \quad j \in \mathbb{N}
$$
is equivalent to
$$
\|U_m - 1\|_{L^{2}((j-1,j) \times \D)} < \tau, \quad \forall m \in \mathbb{N} \quad \mbox{and} \quad \forall j \leq -k_*. 
$$
This inequality combined with (\ref{EQT2*}) gives
$$
\|U - 1\|_{L^{2}((j-1,j) \times \D)} \leq \tau, \quad \forall j \leq -k_* .
$$
Therefore, when $(II)$  occurs, the Proposition  \ref{crucial*} holds with $j_0=k_*$.  The cases $(III)$ and $(IV)$ can be analyzed of the same way, then we omit their proofs, and the proposition is proved. $\Box$

\section{Proof of Theorem \ref{T22}: $A$ verifies the Rabinowitz's condition}

In this section we establish the existence of a heteroclinic solution for the Class 2. In what follows, we are considering the equation  
\begin{equation} \label{E12}
-\Delta u + A(\epsilon x, y)V'(u)=0, \quad \mbox{in} \quad \Omega, \tag{$P_\epsilon$}
\end{equation}
together with the Neumann  boundary condition
\begin{equation} \label{E13}
\frac{\partial u}{\partial \nu}(x,y)=0, \ x \in \R,  y \in \partial \D, \tag{NC}
\end{equation}
where $\epsilon$ is a positive parameter and $A$ satisfies 
$$
0<A_0=A(0,y)=\inf_{(x,y) \in \Omega}A(x,y)\leq \liminf_{|(x,y)| \to +\infty}A(x,y)=A_\infty<\infty, \quad \forall y \in \D.  \leqno{(A_3)}
$$
From now on, we are denoting by  $J_\epsilon, J_\infty:\Gamma \rightarrow \R \cup \{+\infty\}$ the functionals
$$
	J_\epsilon(U)=\sum_{k \in \mathbb{Z}} I_{\epsilon,k} (U)  
$$
and 
$$
	J_\infty(U)=\sum_{k \in \mathbb{Z}} I_{\infty,k} (U), 
$$
where $I_{\epsilon,k},I_{\infty,k}: E_k \rightarrow \R$ are given by
$$
I_{\epsilon, k}(U)=\int_{k}^{k+1} \int_{\D} \left( |\nabla U|^{2}+A(\epsilon x, y)V(U) \right)\, dx dy
$$
and
$$
I_{\infty,k}(U)=\int_{k}^{k+1} \int_{\D} \left( |\nabla U|^{2}+A_\infty V(U) \right)\, dx dy.
$$
Moreover, we denote by $\Theta_{\epsilon}$ and $\Theta_\infty$ the following numbers
$$
\Theta_\epsilon=\inf\{J_\epsilon(U)\,:\,U \in \Gamma\}	
$$	
and
$$
\Theta_\infty=\inf\{J_\infty(U)\,:\,U \in \Gamma\}.	
$$	
	
By Section 2,  we know that there are $W_0, W_\infty \in \Gamma$ verifying $J_0(W_0)=\Theta_0$ and $ J_\infty(W_\infty)=\Theta_\infty$. This fact permit us to prove the following lemma

\begin{lem} $\displaystyle \limsup_{\epsilon \to 0}\Theta_{\epsilon} \leq \Theta_0$ and $\Theta_0 < \Theta_\infty$.
\end{lem}
	
\begin{proof} For each $\epsilon >0$, 
$$
\Theta_{\epsilon} \leq J_\epsilon(W_0).
$$
Since 
$$
\lim_{\epsilon \to 0}J_\epsilon(W_0)=J_0(W_0)=\Theta_0,
$$ 	
it follows that
$$
\limsup_{\epsilon \to 0}\Theta_{\epsilon} \leq \Theta_0.
$$
On the other hand, by $(A_3)$, 
$$
\Theta_0 \leq J_0(W_0)< J_\infty(W_\infty)=\Theta_\infty,
$$
which shows the lemma. 
\end{proof}

In the sequel, we fix $\epsilon_0>0$ small enough a such way that 
\begin{equation} \label{E14}
\Theta_\epsilon < \Theta_\infty, \quad \forall \epsilon \in (0,\epsilon_0).
\end{equation}

\subsection{Proof of Theorem \ref{T22}}

Arguing as in Section 2, for each $\epsilon >0$ there is a minimizing sequence $(U_n) \subset \Gamma$ with $-1 \leq U_n(x,y) \leq 1$ for all $(x,y) \in \Omega$ and $U \in W^{1,2}_{loc}(\Omega)$ such that 
$$
J_\epsilon(U_n) \to \Theta_\epsilon,
$$
$$
U_n \rightharpoonup U \quad \mbox{in} \quad  E_k, \quad \forall k \in \mathbb{Z},
$$
$$
U_n \to U \quad \mbox{in} \quad L^{2}((k,k+1) \times \D), \quad \forall k \in \mathbb{Z},
$$
$$
U_n(x,y) \to U(x,y), \quad \mbox{a.e. in} \quad \Omega,
$$
$$
-1 \leq U(x,y) \leq 1, \quad \forall (x,y) \in \Omega, 
$$
and
\begin{equation} \label{E15}
J_\epsilon(U) \leq \Theta_\epsilon.
\end{equation}

In the sequel, we will use the same approach explored in Section 3. As $(U_n) \subset \Gamma$ and $\tau \in (0, \sqrt{|\Omega_1|})$, for each  $m \in \mathbb{N}$ there are $k_1(m), k_2(m) \in \mathbb{Z}$ such that 
\begin{equation} \label{EE0}
	\|P_{-j}Q_m - 1\|_{L^{2}(\Omega_1)} < \tau, \quad \|Q_m - 1\|_{L^{2}(\Omega_1)} \geq \tau
\end{equation}
and
\begin{equation} \label{EE00}
	\|P_{j}R_m + 1\|_{L^{2}(\Omega_1)} < \tau , \quad \|R_m + 1\|_{L^{2}(\Omega_1)} \geq \tau 
\end{equation}
for all $j \in \mathbb{N}$, where
$$
Q_m(x,y)=U_{m}(x-k_1(m),y) \quad \mbox{and} \quad R_m(x,y)=U_{m}(x+k_2(m),y), \quad \forall (x,y) \in \Omega.
$$

\begin{prop} \label{crucial1} If $\epsilon \in (0, \epsilon_0)$, then there is $j_0 \in \mathbb{N}$ such that 
	\begin{equation} \label{NEWLIMIT}
	\|U-1\|_{L^{2}((-j,-j+1) \times \D)}\leq \tau \quad \mbox{and} \quad  \|U+1\|_{L^{2}((j,j+1) \times \D)}\leq \tau, \quad \forall j \geq j_0.
\end{equation}	
\end{prop}

\begin{proof}  As in the proof of Proposition \ref{crucial*}, we must study the cases $(I)-(IV)$. The case $(I)$ follows of the same way, however for the other cases we need to do some modifications. Next, we will consider the case $(II)$. As in the last section, we begin by showing the claim below 

\begin{claim} \label{NOVACLAIM2}
	$(k_1(m))$ is bounded.
\end{claim}

In what follows we denote $k_1(m)$ by $k(m)$. Assume by contradiction that there is a $\epsilon \in (0, \epsilon_0)$ such that $(k(m))$ is unbounded and $k(m) \to +\infty$. The boundedness of $(U_m)$ in $E_k$ implies that $(Q_m)$ is also bounded in $E_k$ for all $k \in \mathbb{N}$. Thus, for some subsequence, there is $W \in W_{loc}^{1,2}(\Omega)$ such that
	\begin{equation} \label{E16}
	Q_m \rightharpoonup W \quad \mbox{in} \quad E_k, \quad \forall k \in \mathbb{N},
	\end{equation}
	\begin{equation} \label{E17}
	Q_m(x,y) \to W(x,y), \quad \mbox{a.e. in} \quad \Omega,
	\end{equation}
	and
	\begin{equation} \label{E18}
	-1 \leq  W(x,y) \leq 1, \quad \mbox{a.e. in} \quad \Omega.
	\end{equation}

By a simple change variable,
\begin{equation} \label{thetaEp}
	\sum_{k \in \mathbb{Z}} \tilde{I}_k (W_m)\leq J_\epsilon(U_m)=\Theta_{\epsilon}+o_m(1) 
\end{equation}
	where 
	$$
	\tilde{I}_{\epsilon,k}(U)=\int_{k}^{k+1} \int_{\D} \tilde{\L}_{\epsilon,m}(U) dx dy
	$$
	with
	$$
	\tilde{\L}_{\epsilon,m}(u)=\frac{1}{2}|\nabla u|^2 + A(\epsilon x-\epsilon k(m), y)V(u).
	$$
	Now, the Fatou's Lemma combined with (\ref{thetaEp}) leads to
	\begin{equation} \label{E191}
	J_{\infty}(W) \leq \Theta_{\epsilon}.
	\end{equation}	
Then,   
\begin{equation} \label{EWW191}
	I_{\infty,j}(W) \to 0 \quad \mbox{as} \quad j \to +\infty.
\end{equation}
Setting for each $j \in \mathbb{N}$ the function $\widetilde{W}_j=P_{-j}W$, the fact that  $W \in L^{\infty}(\Omega)$ implies that there are $W_0 \in W^{1,2}_{loc}({\Omega}, \mathbb{R})$  and a subsequence of $(\widetilde{W}_j)$, still denoted by itself,  such that
	$$
\widetilde{W}_j \to W_0 \,\, \mbox{in} \,\, W^{1,2}({\Omega_1}) \quad \mbox{as} \quad j \to +\infty,
	$$
	and so,
	\begin{equation} \label{E21}
	\|\widetilde{W}_j - W_0\|_{L^{2}(\Omega_1)} \to 0.
	\end{equation}
This limit combined (\ref{EE0}) yields 
$$
\|W_0 - 1\|_{L^{2}(\Omega_1)} \leq \tau.
$$
On the other hand, by (\ref{EWW191}),
$$
	I_{\infty,0}(W_0)=0, 
$$
which gives $W_0=1$ or $W_0=-1$. 
As $ \tau \in (0,\sqrt{|\Omega_1|})$, we must have	$W_0=1$. Then, 
	\begin{equation} \label{E22}
	\|W_j - 1\|_{L^{2}(\Omega_1)} \to 0 \quad \mbox{as} \quad j \to +\infty.
	\end{equation}

	By using the same type of argument, fixing $W_j=P_jW$ for $j \in \mathbb{N}$, it is possible to prove that there exist $\widehat{W}_0 \in W_{loc}^{1,2}(\overline{\Omega})$  and a subsequence of $(W_j)$, still denoted by itself,  such that
	$$
{W}_j \to \widehat{W}_0 \,\, \mbox{in} \,\, W^{1,2}({\Omega_1}) \quad \mbox{as} \quad j \to -\infty ,
	$$
	and so,
	\begin{equation} \label{E23}
	\|W_j - \widehat{W}_0\|_{L^{2}(\Omega_1)} \to 0.
	\end{equation}
Thereby, $\widehat{W}_0=1$ or $\widehat{W}_0=-1$. Here, as in the previous section, we have that $\widehat{W}_0=-1$. Indeed, assuming by contradiction that $\widehat{W}_0=1$, we set the  function
	$$
	H_j(x,y)=
	\left\{
	\begin{array}{ll}
	1, & \quad x \leq j, y \in \D, \\\
	((j+1)-x)+(x-j)Q_m(x,y), & \quad j < x \leq j+1, y \in \D, \\
	Q_m(x,y), & \quad j+1 < x, y \in \D.
	\end{array}
	\right.
	$$
	Arguing as Proposition \ref{crucial*}, we will find 
	$$
	\Theta_{\infty} \leq \Theta_\epsilon,
	$$
which contradicts (\ref{E14}), and then $\widehat{W}_0=-1$. Now we follow the same idea explored in Proposition \ref{crucial*}  to conclude the proof.
\end{proof}

Now, we are ready to prove Theorem \ref{T22}. \\

\noindent {\bf Proof of Theorem \ref{T22} } \, As an immediate consequence of the last proposition, for each $\epsilon \in (0, \epsilon_0)$, there is $j_0 \in \mathbb{N}$ such that
$$
\|U-1\|_{L^{2}((-j,-j+1) \times \D)}<\tau \quad \mbox{and} \quad  \|U+1\|_{L^{2}((j,j+1) \times \D)}<\tau, \quad \forall j \geq j_0.
$$
Now, arguing as in the proof of Theorem \ref{T1}, it follows that $U \in C^{2}(\overline{\Omega}, \mathbb{R})$. Moreover, $U$ is a classical solution of  
$$	
-\Delta U + A(\epsilon x,y)V'(U)=0, \quad \mbox{in} \quad \Omega \quad  \mbox{and} \quad \frac{\partial U}{\partial \nu}=0, \ x \in \R, \ y \in \partial \D
$$
with 
$$
U(x,y) \to 1 \quad \mbox{as} \quad x \to -\infty \quad \mbox{and} \quad  U(x,y) \to -1 \quad \mbox{as} \quad x \to +\infty, \quad \mbox{uniformly in} \quad y \in \D.
$$
From this, $U$ is a heteroclinic solution from 1 to -1, which finishes the proof of Theorem \ref{T22}. $\Box$
\vspace{0.5 cm}

\noindent {\bf Acknowledgment.}  The author would like to warmly thank Prof. Olimpio Hiroshi Miyagaki for several discussions about this subject, and also to Professor Rabinowitz by his comments that were very important to improve this manuscript. Moreover, the author would like to thank to the referee for his/her very nice remarks and suggestions, which were very important to improve this manuscript.

\end{document}